\documentclass{amsart}
\usepackage{amssymb}
\usepackage{graphicx}

\numberwithin{equation}{section}

\theoremstyle{plain}
\newtheorem{thm}[equation]{Theorem}
\newtheorem{cor}[equation]{Corollary}
\newtheorem{lem}[equation]{Lemma}
\newtheorem{prop}[equation]{Proposition}
\newtheorem{claim}{Claim}

\theoremstyle{definition}

\theoremstyle{remark}
\newtheorem{rem}[equation]{Remark}

\title[Distinct distance three splittings]{A closed orientable 3-manifold with distinct distance three genus two Heegaard splittings}

\author{John Berge}
\email{jberge@charter.net}
\begin{document}

\begin{abstract}
We describe an example of a closed orientable 3-manifold $M$ with distinct distance three genus two Heegaard splittings. This demonstrates that the constructions of alternate genus two Heegaard splittings of closed orientable 3-manifolds described by Rubinstein and Scharlemann in their 1998 paper \cite{RS}, \emph{Genus Two Heegaard Splittings of Orientable 3-Manifolds}, does not yield all alternate genus two splittings, and must therefore be augmented by the constructions described in \cite{BS}.
\end{abstract}

\maketitle

\section{Introduction}
The main result of this paper shows there are closed orientable 3-manifolds which have distinct distance three genus two Heegaard splittings. On the other hand, it is shown in \cite{BS}, which corrects an omission in the 1998 paper \cite{RS} in which Rubinstein and Scharlemann applied their powerful method of sweepouts to the problem of classifying genus two Heegaard splittings of closed orientable 3-manifolds, that all of the constructions of potential alternate genus two Heegaard splittings described in \cite{RS} yield splittings of distance no more than two. It follows that the constructions of \cite{RS} must be augmented by those of \cite{BS} in order to obtain all possible alternate genus two splittings of a closed orientable 3-manifold.

\section{preliminaries} \label{preliminaries}

A Heegaard splitting $(\Sigma;V,W)$, of genus $g$, of a closed orientable 3-manifold $M$ consists of a closed orientable surface $\Sigma$, of genus $g$, and two handlebodies $V$ and $W$, such that $\Sigma$ = $\partial V$ = $\partial W$, $V \cap W = \Sigma$, and $M = V \cup W$. A Heegaard splitting $(\Sigma;V,W)$ is \emph{reducible} if there exists an essential separating simple closed curve in $\Sigma$ that bounds disks in both $V$ and $W$. A splitting $(\Sigma;V,W)$ is \emph{irreducible} if it is not reducible. A set $v$ of pairwise disjoint disks in a handlebody $V$ is a \emph{complete set of cutting disks} for $V$ if cutting $V$ open along the members of $v$ yields a 3-ball. If $v$ and $w$ are complete sets of cutting disks of $V$ and $W$ respectively, the set of simple closed curves $\partial v \cup \partial w$ in $\Sigma$ is a Heegaard diagram $(\Sigma;\partial v, \partial w)$ of $(\Sigma;V,W)$. The \emph{complexity} $c(\Sigma;\partial v, \partial w)$ of $(\Sigma;\partial v, \partial w)$ is the number of points in $\partial v \,\cap\, \partial w$. (We always assume this number has been minimized by isotopies of $\partial v$ and $\partial w$ in $\Sigma$.) 

The complete set of cutting disks $v$ of $V$ \emph{minimizes} the complete set of cutting disks $w$ of $W$ if $$c(\Sigma;\partial v,\partial w) \leq c(\Sigma;\partial v',\partial w),$$ for each complete set of cutting disks $v'$ of $V$. 

The complete set of cutting disks $v$ of $V$ is a set of \emph{universal minimizers} if for any complete sets of cutting disks $v'$ of $V$ and $w$ of $W$ $$c(\Sigma;\partial v,\partial w) \leq c(\Sigma;\partial v',\partial w).$$

A complete set of cutting disks $v$ of $V$ is a set of \emph{strict universal minimizers}, or set of SUMS, if for any complete sets of cutting disks $v'$ of $V$ and $w$ of $W$, with $v' \neq v$ $$c(\Sigma;\partial v,\partial w) < c(\Sigma;\partial v',\partial w).$$

A Heegaard diagram $(\Sigma;\partial v, \partial w)$ is \emph{minimal} if $$c(\Sigma;\partial v,\partial w) \leq c(\Sigma;\partial v',\partial w')$$ whenever $v'$ and $w'$ are complete sets of cutting disks of $V$ and $W$ respectively. (Note that if the complete set of cutting disks $v$ of $V$ is a set of SUMS, and $(\Sigma;\partial v', \partial w)$ is minimal, then $v = v'$.)

Two Heegaard splittings $(\Sigma;V,W)$ and $(\Sigma';V',W')$ of $M$ are \emph{homeomorphic} if there is a homeomorphism $ h: M \to M$ such that $h(\Sigma) = \Sigma'$. Two Heegaard diagrams $(\Sigma;\partial v, \partial w)$ and $(\Sigma';\partial v',\partial w')$ of $M$ are \emph{equivalent} if there is a homeomorphism $ h: \Sigma \to \Sigma'$ such that $h(\partial v) = \partial v'$ and $h(\partial w) = \partial w'$. 
Note that if $(\Sigma;\partial v, \partial w)$ and $(\Sigma';\partial v',\partial w')$ are equivalent Heegaard diagrams of $M$, then $(\Sigma;V,W)$ and $(\Sigma';V',W')$ are equivalent Heegaard splittings of $M$.

Finally, two splittings $(\Sigma;V,W)$ and $(\Sigma';V',W')$ of $M$ are \emph{isotopic} if there is an ambient isotopy of $M$ which carries $\Sigma$ to $\Sigma'$.

\begin{rem}
Methods for detecting the presence of a set of SUMS in one of the handlebodies of a genus two Heegaard splitting of a closed orientable 3-manifold $M$, arguments that the existence of a set of SUMS is a generic condition among genus two Heegaard splittings, and applications of the existence of a detectable set of SUMS to the problem of determining all alternative genus two splittings of $M$ will appear elsewhere.
\end{rem}

\begin{rem}
Note that we make extensive use of R-R diagrams. See \cite{B1} for some background material on these. 
\end{rem}

\subsection{The forms of graphs underlying genus two Heegaard splittings} \hfill 
\smallskip

It will be helpful to have a list of the types of graphs which can underlie genus two Heegaard diagrams. Figure \ref{Dist3Fig8b} displays the possible graphs. Lemmas \ref{3 types of graphs which occur} and \ref{3 types of graphs which don't occur} show these are the only possibilities. (We note a version of Lemma~\ref{3 types of graphs which occur} appears in \cite{HOT} where it is credited to \cite{O}.)

\begin{lem} \label{3 types of graphs which occur}
If $W$ is a genus two handlebody with a complete set of cutting disks $\{D_S,D_T\}$, and $\mathcal{C}$ is a set of disjoint essential simple closed curves in $\partial W$ such that each curve in $\mathcal{C}$ has only essential intersections with $D_S$ and $D_T$, and no curve in $\mathcal{C}$ is disjoint from both $D_S$ and $D_T$, then the Heegaard diagram of the curves in $\mathcal{C}$ with respect to $\{D_S,D_T\}$ has a graph $G_\mathcal{C}$ with the form of one of the three graphs in Figure \emph{\ref{Dist3Fig8b}}.
\end{lem}

\begin{proof}
It is easy to enumerate the possibilities here using the result of Lemma \ref{3 types of graphs which don't occur}, which shows that all of the edges of $G_\mathcal{C}$ which connect a vertex of $G_\mathcal{C}$ corresponding to a given side of $D_S$ with a vertex of $G_\mathcal{C}$ corresponding to a given side of $D_T$, must be parallel.
\end{proof}

\begin{lem} \label{3 types of graphs which don't occur}
Suppose the hypotheses of Lemma \emph{\ref{3 types of graphs which occur}} hold. Then any two edges of $G_\mathcal{C}$ connecting $S^+$ \emph{(}resp $S^-$\emph{)} to $T^+$ \emph{(}resp $T^-$\emph{)} must be parallel.
\end{lem}

\begin{proof}
Suppose, to the contrary, that there are nonparallel edges in $G_\mathcal{C}$ connecting, say, $S^+$ to $T^+$. Then it is easy to see that $G_\mathcal{C}$ must have the form of one of the three graphs in Figure \ref{Dist3Fig8c} with $a > 0$ and $b > 0$. 
Next, if $v \in \{S^+,S^-,T^+,T^-\}$ is a vertex of $G_\mathcal{C}$, let $V(v)$ be the number of ends of edges of $G_\mathcal{C}$ which meet $v$. Then, if $G_\mathcal{C}$ is a graph underlying a genus two Heegaard diagram, the equations $V(S^+)$ = $V(S^-)$ and $V(T^+)$ = $V(T^-)$ must hold. However, one checks easily that these equations do not hold in Figures \ref{Dist3Fig8c}a, \ref{Dist3Fig8c}b or \ref{Dist3Fig8c}c unless $a = b = 0$.
\end{proof}

\begin{figure}[htbp]
\centering
\includegraphics[width = 1.0\textwidth]{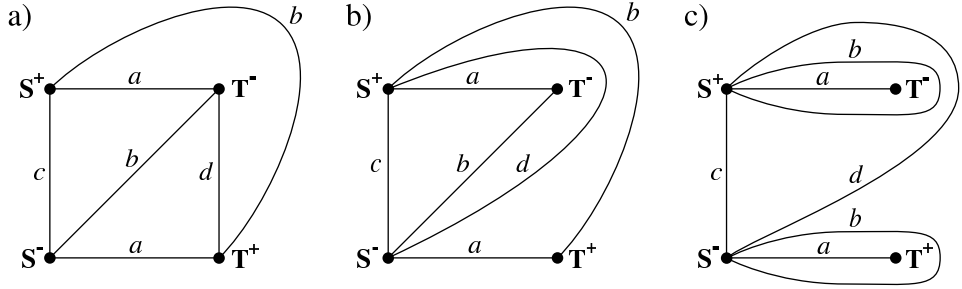}
\caption{If $W$ is a genus two handlebody with a complete set of cutting disks $\{D_S,D_T\}$, and $\mathcal{C}$ is a set of disjoint essential simple closed curves in $\partial W$ such that each curve in $\mathcal{C}$ has only essential intersections with $D_S$ and $D_T$, and no curve in $\mathcal{C}$ is disjoint from both $D_S$ and $D_T$, then the Heegaard diagram of the curves in $\mathcal{C}$ with respect to $\{D_S,D_T\}$ has a graph $G_\mathcal{C}$ with the form of one of these three graphs.}
\label{Dist3Fig8b}
\end{figure}

\begin{figure}[htbp]
\centering
\includegraphics[width = 1.0\textwidth]{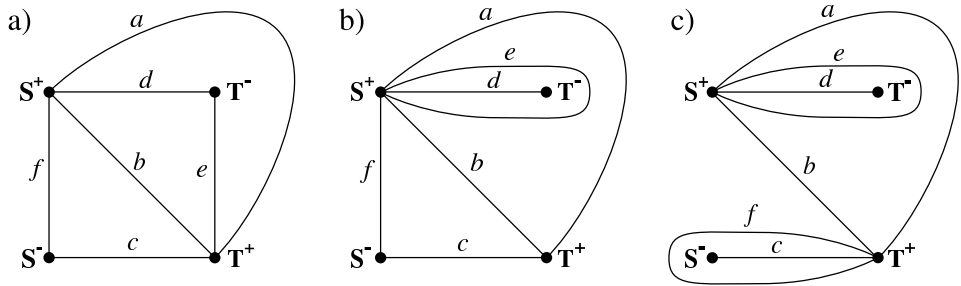}
\caption{If $W$, $\{D_S,D_T\}$, and $\mathcal{C}$ are as in Figure \ref{Dist3Fig8b}, then the Heegaard diagram of the curves in $\mathcal{C}$ with respect to $\{D_S,D_T\}$ does not have a graph $G_\mathcal{C}$ with the form of one of the three graphs in this figure unless $a = b = 0$.}
\label{Dist3Fig8c}
\end{figure}

\subsection{The distance of a Heegaard splitting} \hfill
\smallskip

Suppose $\Sigma$ is a Heegaard surface in a closed orientable 3-manifold $M$ in which $\Sigma$ bounds handlebodies $H$ and $H'$.
The Heegaard splitting of $M$ by $\Sigma$ is a splitting of \emph{distance} $n$ if there is a sequence $c_0, \dots, c_n$ of essential simple closed curves in $\Sigma$ such that:
\begin{enumerate}
	\item $c_0$ bounds a disk in $H$;
	\item $c_n$ bounds a disk in $H'$; 
	\item if $n > 0$, $c_i$ and $c_{i+1}$ are disjoint for $0 \leq i < n$; 
	\item $n$ is the smallest nonnegative integer such that (1), (2) and (3) hold.
\end{enumerate}

A Heegaard splitting of distance $0$ is \emph{reducible}. A splitting of distance
$1$ is \emph{weakly reducible}. Any Heegaard splitting of a reducible manifold is reducible. A Heegaard splitting of distance at most $2$ has the \emph{disjoint curve property} or DCP \cite{Th}; i.e. there is an essential nonseparating simple closed curve in $\Sigma$ which is disjoint from nonseparating properly embedded disks in both $H$ and $H'$. Any Heegaard splitting of a toroidal 3-manifold has the DCP (\cite{He},\cite{Th}). A weakly reducible genus two Heegaard splitting is also reducible, so an irreducible Heegaard splitting of genus two has distance at least two \cite{Th}.

\begin{rem}
If a closed orientable 3-manifold has a Heegaard splitting with distance at
least three, then it is irreducible, atoroidal, and it is not a Seifert manifold by Hempel \cite{He}, so by Perelman's proof of Thurston's Geometrization Conjecture, the manifold is hyperbolic.
\end{rem}

\section{Alternate genus two Heegaard splittings and (SF,PP) pairs}

Suppose $H$ is a genus two handlebody, and $\alpha$ is a nonseparating simple closed curve in $\partial H$. The curve $\alpha$ is \emph{Seifert Fiber} or \emph{SF} in $H$ if attaching a 2-handle to $H$ along $\alpha$ yields an orientable Seifert fibered space over the disk $D^2$ with 2 exceptional fibers.
A nonseparating simple closed curve $\beta$ in $\partial H$ is \emph{primitive} in $H$ if there exists a disk $D$ in $H$ such that $|\beta \cap D| = 1$. Equivalently $\beta$ is conjugate to a free generator of $\pi_1(H)$. The curve $\beta$ is a \emph{proper power} or \emph{PP} in $H$ if $\beta$ is disjoint from a separating disk in $H$, $\beta$ does not bound a disk in $H$, and $\beta$ is not primitive in $H$. 
A pair of disjoint nonseparating simple closed curves $(\alpha,\beta)$ in the boundary of a genus two handlebody $H$ is a $(Seifert Fiber, Proper Power)$ pair, or $(SF,PP)$ pair if attaching a 2-handle to $H$ along $\alpha$ yields an orientable Seifert fibered space over the disk $D^2$ with 2 exceptional fibers, and $\beta$ is a proper power of a free generator of $\pi_1(H)$.

\begin{rem}
Suppose $H$ is a genus two handlebody. Due to the work of Zieschang and others, nonseparating simple closed curves in $\partial H$ which are $SF$ curves are completely understood. (See the expository paper \cite{Z2} of Zieschang and its excellent bibliography.) Using this classification, it is not hard to show that if $\alpha$ is $SF$ in $\partial H$, then there exists a nonseparating curve $\beta$ disjoint from $\alpha$ such that $(\alpha, \beta)$ is a $(SF,PP)$ pair in $\partial H$, and the pair $(\alpha,\beta)$ has an R-R diagram with the form of Figure \ref{Dist3Fig16}a or \ref{Dist3Fig16}b.
\end{rem}

\begin{figure}[htbp]
\centering
\includegraphics[width = 1.0\textwidth]{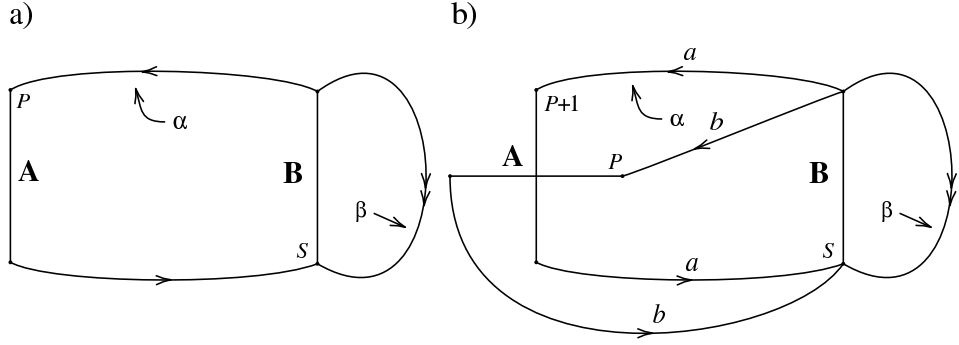}
\caption{If $(\alpha, \beta)$ is a $(SF,PP)$ pair on the boundary of a genus two handlebody $H$, then $\alpha$ and $\beta$ have an R-R diagram on $\partial H$ with the form of Figure \ref{Dist3Fig16}a or the form of Figure \ref{Dist3Fig16}b. Here, $|P|, |S| > 1$, $a, b > 0$, and $\gcd(a,b) = 1$.}
\label{Dist3Fig16}
\end{figure}

The following theorem explains our interest in $(SF,PP)$ pairs.

\begin{thm} \label{Alternative splittings from (SF,PP), (PP,SF) pairs}
Suppose $\Sigma$ is a genus two Heegaard surface bounding handlebodies $H$ and $H'$ in a closed orientable 3-manifold $M$. If $\alpha$ and $\beta$ are disjoint nonseparating simple closed curves in $\Sigma$ such that $(\alpha,\beta)$ is a $(SF,PP)$ pair in $H$ and a $(PP,SF)$ pair in $H'$, then $M$ has an alternative genus two Heegaard surface $\Sigma'$, which is not obviously isotopic to $\Sigma$.
\end{thm}

\begin{proof}
Let $N_\alpha$ and $N_\beta$ be disjoint regular neighborhoods of $\alpha$ and $\beta$ respectively in $\Sigma$. Since $\beta$ is a proper power in $H$, the boundary components of $N_\beta$ bound an essential separating annulus $\mathcal{A}_\beta$ in $H$, and cutting $H$ open along $\mathcal{A}_\beta$ cuts $H$ into a genus two handlebody $H_\beta$ and a solid torus $V_\beta$. Similarly, since $\alpha$ is a proper power in $H'$, the boundary components of $N_\alpha$ bound an essential separating annulus $\mathcal{A}_\alpha$ in $H'$, and cutting $H'$ open along $\mathcal{A}_\alpha$ cuts $H'$ into a genus two handlebody $H'_\alpha$ and a solid torus $V'_\alpha$.

Then $N_\alpha$ lies in $\partial H_\beta$, and $\alpha$ is primitive in $H_\beta$. And similarly, $N_\beta$ lies in $\partial H'_\alpha$, and $\beta$ is primitive in $H'_\alpha$. 
To see that $\alpha$ is primitive in $H_\beta$, let $H[\alpha]$ denote the manifold obtained by adding a 2-handle to $\partial H$ along $\alpha$. By hypothesis, $H[\alpha]$ is Seifert fibered over $D^2$ with two exceptional fibers. The annulus $\mathcal{A}_\beta$ is an essential separating annulus in $H[\alpha]$, which must be vertical in the Seifert fibration of $H[\alpha]$. This implies that $H_\beta[\alpha]$ is Seifert fibered over the disk $D^2$ with one exceptional fiber. So $H_\beta[\alpha]$ is a solid torus. This can occur only if $\alpha$ is primitive in $H_\beta$. Similarly, $\beta$ is primitive in $H_\alpha$.

Returning to the main argument, $N_\alpha$ lies in $\partial V'_\alpha$, and $N_\beta$ lies in $\partial V_\beta $. It follows that $H'_\alpha \cup_{N_\beta} V_\beta$ and $H_\beta \cup_{N_\alpha} V'_\alpha$ are each genus two handlebodies. Their common boundary is then an alternative genus two Heegaard surface $\Sigma'$ for $M$.
\end{proof}

\begin{rem}
The type of alternative genus two Heegaard splittings described in Theorem \ref{Alternative splittings from (SF,PP), (PP,SF) pairs}, which arise from $(\alpha,\beta)$ pairs in genus two Heegaard surfaces that are $(SF,PP)$ pairs in one of the handlebodies bounded by the surface and $(PP,SF)$ pairs in the other handlebody bounded by the surface, are exactly those which were overlooked in the classification  \cite{RS}.
\end{rem}

\newpage

\section{An example of distinct distance three genus two splittings}

\begin{figure}[htbp]
\centering
\includegraphics[width = 0.75\textwidth]{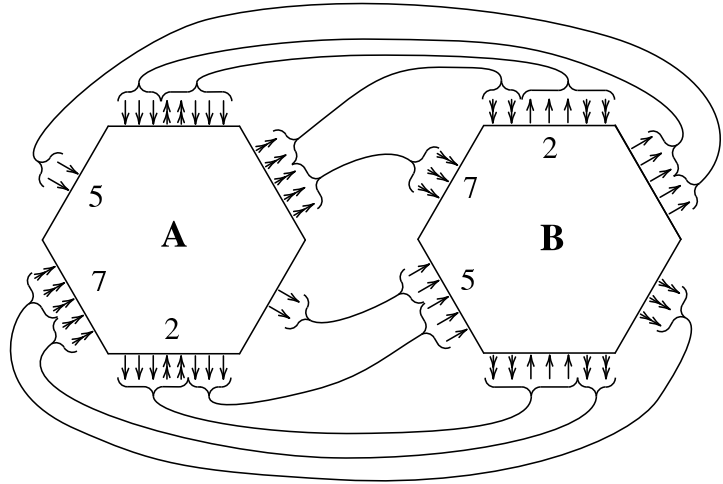}
\caption{An R-R diagram of the first Heegaard splitting of $M$.}
\label{Dist3Fig9a}
\end{figure}

\begin{figure}[htbp]
\centering
\includegraphics[width = 0.75\textwidth]{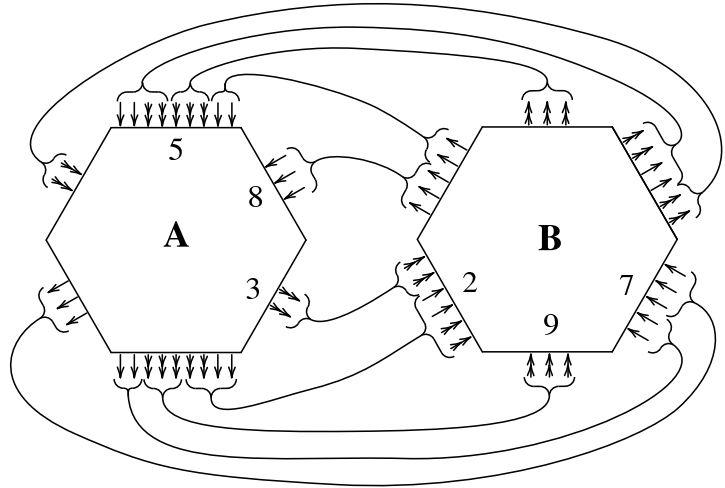}
\caption{An R-R diagram of the second splitting of $M$.} 
\label{Dist3Fig9b}
\end{figure}

\begin{figure}[htbp]
\centering
\includegraphics[width = 0.80\textwidth]{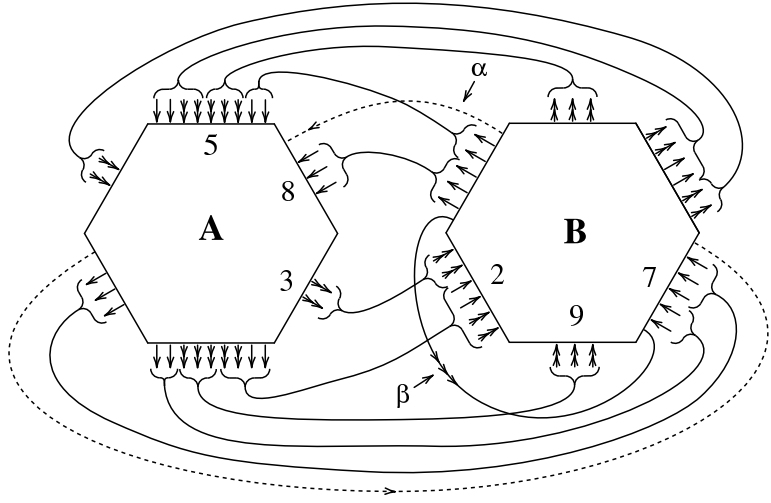}
\caption{The R-R diagram of the second splitting of $M$ from Figure~\ref{Dist3Fig9b} with an $(\alpha,\beta)$ pair which is $(SF,PP)$ in $H$ and $(PP,SF)$ in $H'$ added to the diagram. Here $(\alpha,\beta)$ = $(A^8B^7,B^7)$ in $\pi_1(H)$, while $(\alpha,\beta)$ = $(x^2,X^2Y^7)$ in $\pi_1(H')$.}
\label{Dist3Fig9d}
\end{figure}

\begin{thm} \label{main result}
The R-R diagrams in Figures \emph{\ref{Dist3Fig9a}} and \emph{\ref{Dist3Fig9b}} represent distinct genus two Heegaard splittings, each of distance \emph{$3$}, of a closed orientable 3-manifold $M$.
\end{thm}

\begin{proof}
Corollary \ref{Splittings are distance three} shows the Heegaard splittings described by the R-R diagrams in Figures \ref{Dist3Fig9a} and \ref{Dist3Fig9b} are each distance three splittings, while Section \ref{Obtaining the second splitting} shows the splittings are both splittings of the same closed orientable 3-manifold $M$. Then Proposition \ref{splittings are not homeomorphic} finishes the proof by showing the splittings of $M$ described by the R-R diagrams in Figures \ref{Dist3Fig9a} and \ref{Dist3Fig9b} are not homeomorphic.
\end{proof}

Proofs of the results leading to Theorem~\ref{main result} occupy the bulk of the remainder of the paper. However, we start with two preliminary subsections. The first of these, Subsection~\ref{R-R diagram explanation} explains the R-R diagrams used in the remainder of the paper, such as Figures~\ref{Dist3Fig9a} and \ref{Dist3Fig9b}. The second subsection, Subsection~\ref{Rectangles in Sigma}, describes certain rectangles in the Heegaard surfaces of Figures~\ref{Dist3Fig9a} and \ref{Dist3Fig9b}, which are later used to show that sets of SUMS exist.

\subsection{A word about the R-R diagrams in Figures \ref{Dist3Fig9a} and \ref{Dist3Fig9b}} \hfill \label{R-R diagram explanation}
\smallskip

The R-R diagrams appearing in this and following sections differ slightly from those described in \cite{B1}, and those appearing in previous sections. This subsection aims to explain these diagrams.

Suppose $\Sigma$ is a genus two Heegaard surface bounding handlebodies $H$ and $H'$ in a closed orientable 3-manifold $M$, $\{D_A,D_B\}$ and $\{D_X,D_Y\}$ are complete sets of cutting disks of $H$ and $H'$ respectively, and $\Gamma$ is an essential separating simple closed curve in $\Sigma$ disjoint from $\partial D_A$ and $\partial D_B$. (Thus $\Gamma$ bounds a disk in $H$ which separates $D_A$ and $D_B$.) In addition, suppose $\partial D_X$ and $\partial D_Y$ have only essential intersections with $\partial D_A$, $\partial D_B$ and $\Gamma$, and both $D_X \cap \Gamma$ and $D_Y \cap \Gamma$ are nonempty. Finally, let $\mathcal{A}$ be a regular neighborhood of $\Gamma$ in $\Sigma$ chosen so that $\mathcal{A}$ is disjoint from $\partial D_A$ and $\partial D_B$, and let $F_A$ and $F_B$ be the two once-punctured tori components of $\Sigma - \text{int}(\mathcal{A})$ with $F_A$ and $F_B$ labeled so that $\partial D_A \subset F_A$ and $\partial D_B \subset F_B$. 

Next, let $\mathcal{S}$, $\mathcal{C}_A$, and $\mathcal{C}_B$ be the sets of arcs $(\partial D_X \cup \partial D_Y) \cap \mathcal{A}$, $(\partial D_X \cup \partial D_Y) \cap F_A$, and $(\partial D_X \cup \partial D_Y) \cap F_B$ respectively. Then R-R diagrams like Figures~\ref{Dist3Fig9a} and \ref{Dist3Fig9b} display the annulus $\mathcal{A}$, minus a point at infinity together with the arcs of $\mathcal{S}$ embedded in the plane $\mathbb{R}^2$ as the closure of the complement of a pair of disjoint hexagons $H_A$ and $H_B$.

This results in the identification of the boundaries $\partial F_A$ and $\partial F_B$ of the once-punctured tori $F_A$ and $F_B$ with $\partial H_A$ and $\partial H_B$ respectively. Let $G \in \{A,B\}$, and let $p$ and $q$ be two points of $(\partial D_X \cup \partial D_Y) \cap \partial F_G$. Then the identification of $\partial F_G$ with $\partial H_G$ has the following properties:

\begin{itemize}
\item The points $p$ and $q$ lie in the same face of $H_G$ if and only if $p$ and $q$ are endpoints of connections $\delta_p$ and $\delta_q$ respectively in $\mathcal{C}_G$ such that $\delta_p$ and $\delta_q$ are properly isotopic in $F_G$ under an isotopy that carries $p$ to $q$.

\item The points $p$ and $q$ lie in opposite faces of $H_G$ if and only if $p$ and $q$ are endpoints of connections $\delta_p$ and $\delta_q$ respectively in $\mathcal{C}_G$ such that $\delta_p$ and $\delta_q$ are properly isotopic in $F_G$ under an isotopy that does not carry $p$ to $q$.

\item Suppose $p$ and $q$ lie in opposite faces of $H_G$, and let $f_p$ and $f_q$ be the faces of $H_G$ containing $p$ and $q$ respectively. Then $|\mathcal{S} \cap f_p|$ = $|\mathcal{S} \cap f_q|$ = $n$, for some nonnegative integer $n$, and there exists a unique set $\Delta$ of $n$ disjoint properly embedded arcs in $H_G$ such that each member of $\Delta$ connects a point of $\mathcal{S} \cap f_p$ to a point of $\mathcal{S} \cap f_q$. Then $p$ and $q$ are endpoints of a connection in $\mathcal{C}_G$ if and only if $p$ and $q$ are connected by an arc in $\Delta$.
\end{itemize}

Once the proper isotopy classes of the connections $\mathcal{C}_G$ in $F_G$ have been determined as above, the only remaining problem is to specify the isotopy class of the simple closed curve $\partial D_G$ in $F_G$. We do this by putting a set of three integer labels next to three consecutive faces of $H_G$; so that one member of each pair of opposite faces of $H_G$ is labeled, and we interpret these integers as algebraic intersection numbers of oriented connections with an oriented simple closed curve $\partial D_G$. This is enough to completely specify the isotopy class of $\partial D_G$ in $F_G$. We note that, a priori, the three consecutive labels can be any 3-triple of integers of the form $(m,m+n,n)$ with $\gcd(m,n) = 1$; so that only two labels would suffice. However, three labels are often convenient.

Finally, suppose $f_A$ and $f_B$ are faces of $H_A$ and $H_B$ respectively, and let $\mathcal{S}' \subset \mathcal{S}$ be the set of arcs in $\mathcal{S}$ which connect points in $f_A$ to points in $f_B$. Then the arcs in $\mathcal{S}$ are parallel in $\mathcal{A}$, and in order to reduce the number of arcs which are displayed in an R-R diagram like Figure~\ref{Dist3Fig9a} or \ref{Dist3Fig9b}, we often group the arcs of $\mathcal{S}'$ together with brackets, which are then connected by a singe arc in $\mathcal{A}$.

(Figure \ref{Dist3Fig7a} illustrates most of the points mentioned above.)

\begin{figure}[htbp]
\centering
\includegraphics[width = 1.00\textwidth]{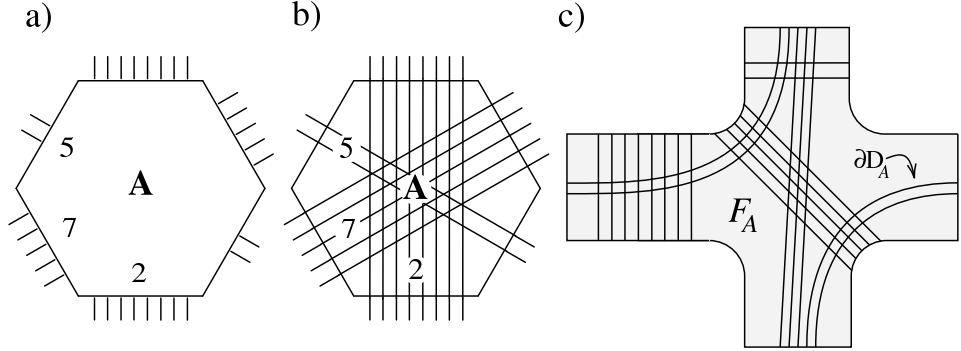}
\caption{Figures illustrating how the identification of $\partial F_A$ with the boundary of a hexagon in Figure \ref{Dist3Fig7a}a, together with the set of three integers placed near three consecutive faces of the hexagon in Figure \ref{Dist3Fig7a}a, encodes the embedding of the simple closed curve $\partial D_A$ and the connections of $(\partial D_X \cup \partial D_Y) \cap F_A$ in $F_A$ shown in Figure \ref{Dist3Fig7a}c. Here Figure~\ref{Dist3Fig7a}c shows the once-punctured torus $F_A$ cut open along a pair of essential arcs parallel to connections in $(\partial D_X \cup \partial D_Y) \cap F_A$. And Figure~\ref{Dist3Fig7a}b shows how points on opposite faces of the hexagon in Figure~\ref{Dist3Fig7a}a are joined by connections in $F_A$.}
\label{Dist3Fig7a}
\end{figure}

\subsection{Rectangles in $\boldsymbol{\Sigma}$} \label{Rectangles and SUMS} \hfill \label{Rectangles in Sigma}
\smallskip

The proofs that each of the Heegaard splittings described by the R-R diagrams of Figures \ref{Dist3Fig9a} and \ref{Dist3Fig9b} is a distance 3 splitting, and that the splittings of Figures~\ref{Dist3Fig9a} and \ref{Dist3Fig9b} are not homeomorphic, depends on the existence of certain rectangles in the genus two Heegaard surfaces of these diagrams. This subsection describes the rectangles we need.

Suppose $\Sigma$ is a genus two Heegaard surface in a closed orientable 3-manifold $M$ such that $\Sigma$ bounds genus two handlebodies $H$ and $H'$, $\{D_A,D_B\}$ and $\{D_X,D_Y\}$ are complete sets of cutting disks of $H$ and $H'$ respectively, and $\Gamma$ is an essential separating curve in $\Sigma$ which bounds a disk in $H$ separating $D_A$ and $D_B$. 

Assuming, as we may, that $\partial D_X$ and $\partial D_Y$ have only essential intersections with $\partial D_A$, $\partial D_B$ and $\Gamma$, suppose $\Gamma$ intersects both $\partial D_X$ and $\partial D_Y$. Then the curves $\Gamma$, $\partial D_X$ and $\partial D_Y$, cut $\Sigma$ into sets of faces, which are either four-sided, i.e. \emph{rectangles}, or have more than four sides. Let $\mathcal{R}$ denote the set of rectangles cut from $\Sigma$ by $\Gamma$, $\partial D_X$ and $\partial D_Y$. We are interested in four subsets of $\mathcal{R}$, which we denote by $\mathcal{R}_{ax}$, $\mathcal{R}_{ay}$, $\mathcal{R}_{bx}$, and $\mathcal{R}_{by}$. The meaning of the subscripts of these sets is as follows: The first letter of the 2-letter subscript $pq$ of the subset $\mathcal{R}_{pq}$ of $\mathcal{R}$ is $a$ (resp $b$) if each rectangle in $\mathcal{R}_{pq}$ lies on the same side of $\Gamma$ in $\Sigma$ as $\partial D_A$ (resp $\partial D_B$). The second letter of the 2-letter subscript $pq$ of the subset $\mathcal{R}_{pq}$ of $\mathcal{R}$ is $x$ (resp $y$) if each rectangle in $\mathcal{R}_{pq}$ has two subarcs of $\partial D_X$ (resp $\partial D_Y$) in its boundary. In addition, each rectangle $R_{pq}$ in $\mathcal{R}_{pq}$ intersects $\partial D_P$ in a number of essential arcs. In each such case, let $|R_{pq}|$ be the number of essential arcs in $R_{pq} \cap \partial D_P$. (It is possible that an $\mathcal{R}_{pq}$ is empty.)

Finally, we mention that a rectangle $R_{pq}$ with $p \in \{a,b\}$, $q \in \{x,y\}$, and $|R_{pq}| = e$ exists in the Heegaard surface of the R-R diagram in Figure~\ref{Dist3Fig9a} (resp Figure~\ref{Dist3Fig9b}), if and only if $\partial D_Q$ intersects the face of the $P$-hexagon in Figure~\ref{Dist3Fig9a} (resp Figure~\ref{Dist3Fig9b}), with label $e$ in two adjacent points.

\begin{thm} \label{4 pairs of rectangles}
Suppose that for each $\mathcal{R}_{pq} \in \{\mathcal{R}_{ax}, \mathcal{R}_{ay}, \mathcal{R}_{bx}, \mathcal{R}_{by}\}$ there exist rectangles $R_{pq_1} \in \mathcal{R}_{pq}$ and $R_{pq_2} \in \mathcal{R}_{pq}$ such that $|R_{pq_1}| - 1 > |R_{pq_2}| > 1$. Then the set of cutting disks $\{D_A,D_B\}$ of $H$ is a set of SUMS, and the Heegaard splitting has no disjoint curves.
\end{thm}

\begin{proof}
Suppose $C_1$ and $C_2$ are a pair of disjoint nonseparating simple closed curves in $\Sigma$ such that $C_1$ bounds a disk in $H'$. We may assume $C_1$ and $C_2$ have only essential intersections with $\partial D_X$, $\partial D_Y$, $\partial D_A$, $\partial D_B$ and $\Gamma$. 

Consider the curve $C_1$. One possibility is that $C_1 \cap (\partial D_X \cup \partial D_Y)$ is nonempty. Suppose this is the case. Then Lemma \ref{forms of graphs of curves bounding disks and disjoint from disks} shows the graph $G(D_X,D_Y\,|\,C_1)$ of the Heegaard diagram of $C_1$ with respect to $D_X$ and $D_Y$ has the form of Figure \ref{Dist3Fig8b}c with $b > 0$, and with the pairs of vertices $\{S^+,S^-\}$ and $\{T^+,T^-\}$ of Figure \ref{Dist3Fig8b}c replaced by $\{X^+,X^-\}$ and $\{Y^+,Y^-\}$. 

Observe that if $\{S^+,S^-\}$ = $\{X^+,X^-\}$ and $\{T^+,T^-\}$ = $\{Y^+,Y^-\}$ in Figure \ref{Dist3Fig8b}c, then $C_1$ intersects every rectangle in $\mathcal{R}_{ay} \cup \mathcal{R}_{by}$ in an essential arc. On the other hand, if $\{S^+,S^-\}$ = $\{Y^+,Y^-\}$ and $\{T^+,T^-\}$ = $\{X^+,X^-\}$ in Figure \ref{Dist3Fig8b}c, then $C_1$ intersects every rectangle in $\mathcal{R}_{ax} \cup \mathcal{R}_{bx}$ in an essential arc.

The remaining possibility is that $C_1$ is disjoint from $\partial D_X \cup \partial D_Y$. In this case, either (2) or (3) of Lemma \ref{forms of graphs of curves bounding disks and disjoint from disks} applies. If $C_1$ is isotopic to $\partial D_X$, then $C_1$ intersects every rectangle in $\mathcal{R}_{ax} \cup \mathcal{R}_{bx}$ in an essential arc. If $C_1$ is isotopic to $\partial D_Y$, then $C_1$ intersects every rectangle in $\mathcal{R}_{ay} \cup \mathcal{R}_{by}$ in an essential arc. 
If $C_1$ is not isotopic to $\partial D_X$ or $\partial D_Y$, then, since $C_1$ is nonseparating in $\Sigma$, $C_1$ separates $X^+$ from $X^-$ and $Y^+$ from $Y^-$ in the graph $G(D_X,D_Y\,|\,C_1)$ of the Heegaard diagram of $C_1$ with respect to $D_X$ and $D_Y$. So, in this case, $C_1$ intersects every rectangle in $\mathcal{R}_{ay} \cup \mathcal{R}_{by}$, and every rectangle in $\mathcal{R}_{ax} \cup \mathcal{R}_{bx}$ in an essential arc.

Next we turn attention to the $H$ side of $\Sigma$. Here the simple closed curve $\Gamma$ cuts $\Sigma$ into two once-punctured tori, $F_A^+$ and $F_B^+$ such that $\partial D_A \subset F_A \subset F_A^+$, $\partial D_B \subset F_B \subset F_B^+$, and $\partial F_A^+$ = $\partial F_B^+$ = $\Gamma$. 

Let $P$ be either $A$ or $B$, and consider the once-punctured torus $F_P^+$. We have just observed that either $C_1$ intersects every rectangle in $\mathcal{R}_{ax} \cup \mathcal{R}_{bx}$ in an essential arc, or $C_1$ intersects every rectangle in $\mathcal{R}_{ay} \cup \mathcal{R}_{by}$ in an essential arc. If $C_1$ intersects every rectangle in $\mathcal{R}_{ax} \cup \mathcal{R}_{bx}$ in an essential arc, let $q = x$; otherwise, let $q = y$. In either case, the hypothesis of Theorem \ref{4 pairs of rectangles} guarantees there exists a pair of rectangles $R_{pq_1} \in \mathcal{R}_{pq}$ and $R_{pq_2} \in \mathcal{R}_{pq}$ such that $|R_{pq_1}| -1 > |R_{pq_2}| > 1$. Let $m$ = $|R_{pq_1}|$, and let $n$ = $|R_{pq_2}|$. Then $m -1 > n > 1$, and the configuration of $R_{pq_1}$, $R_{pq_2}$ and $\partial D_P$ in $F_P^+$ must be homeomorphic to that shown in Figure \ref{Dist3Fig22b}. Next, consider the set of connections $C_1 \cap F_P^+$, and observe that, since $C_1$ intersects both $R_{pq_1}$ and $R_{pq_2}$ in essential arcs, there exist connections $\omega_1$ and $\omega_2$ in $C_1 \cap F_P^+$ such that $\omega_1 \subset R_{pq_1}$ and $\omega_2 \subset R_{pq_2}$.

It is time to consider $C_2$. Note first that, because $C_1$ and $C_2$ are disjoint, and both $C_1 \cap F_A^+$ and $C_1 \cap F_B^+$ contain nonisotopic pairs of connections, $C_2$ can not lie completely in $F_A^+$ or $F_B^+$. So $C_2 \cap F_P^+$ is a set of connections in $F_P^+$. It follows that if $\delta$ is any connection in $C_2 \cap F_P^+$, then $\delta$ is properly isotopic in $F_P^+$ to one of the four connections $\delta_1$, $\delta_2$, $\delta_3$, $\delta_4$ shown in Figure \ref{Dist3Fig22b}. Then, since $m-1 > n > 1$, $\delta_1$, $\delta_2$, $\delta_3$ and $\delta_4$ intersect $\partial D_P$ respectively $m > 3$, $n > 1$, $m+n > 4$ and $m-n > 1$ times. 
In particular, each connection $\delta \in C_2 \cap F_P^+$ satisfies $|\delta \cap \partial D_P| \geq 2$. 

\begin{claim} \label{No disjoint curves in Sigma}
There are no disjoint curves in $\Sigma$. 
\end{claim} 
\begin{claim}\label{DA, DB is a set of SUMS}
The set of cutting disks $\{D_A,D_B\}$ of $H$ is a set of SUMS.
\end{claim}
\begin{proof} [Proof of Claim \emph{\ref{No disjoint curves in Sigma}}]
Since $C_2$ is an arbitrary nonseparating simple closed curve in $\Sigma$ disjoint from a disk in $H'$, it is enough to show that $C_2$ has essential intersections with every cutting disk of $H$. Lemma \ref{forms of graphs of curves bounding disks and disjoint from disks} shows that if $C_2$ is disjoint from a disk in $H$, then the graph $G(D_A,D_B \,|\, C_2)$ of the Heegaard diagram of $C_2$ with respect to $D_A$ and $D_B$ either has no edges connecting $A^+$ to $A^-$, or it has no edges connecting $B^+$ to $B^-$.

But, this is not the case. Since $C_2$ does not lie completely in $F_A^+$, or completely in $F_B^+$, there is a connection $\delta_A \in C_2 \cap F_A^+$ such that $|\delta_A \cap \partial D_A| \geq 2$, and there is a connection $\delta_B \in C_2 \cap F_B^+$ such that $|\delta_B \cap \partial D_B| \geq 2$. This implies there exist edges in 
$G(D_A,D_B \,|\, C_2)$ connecting $A^+$ to $A^-$, and there exist edges in $G(D_A,D_B \,|\, C_2)$ connecting $B^+$ to $B^-$. It follows that $C_2$ is not disjoint from a disk in $H$, and so there are no disjoint curves in $\Sigma$. This proves Claim \ref{No disjoint curves in Sigma}.
\end{proof}

\begin{proof} [Proof of Claim \emph{\ref{DA, DB is a set of SUMS}}]
Now suppose that in addition to being disjoint from $C_1$, $C_2$ bounds a cutting disk of $H'$. Then we may assume $C_1$ and $C_2$ bound an arbitrary complete set of cutting disks of $H'$. And then, by Lemma \ref{c > a+b > 0 and d > a+b > 0 imply minimal}, the complete set of cutting disks $\{D_A,D_B\}$ of $H$ will be a set of SUMS provided we can show that the graph $G(D_A,D_B \,|\, C_1,C_2)$ of the Heegaard diagram of $C_1$ and $C_2$ with respect to $D_A$ and $D_B$ has the form of Figure \ref{Dist3Fig8b}a with $c > a+b > 0$ and $d > a+b > 0$. We proceed to do this.

It was shown in the proof of Claim \ref{No disjoint curves in Sigma} that $G(D_A,D_B \,|\, C_2)$ has edges connecting $A^+$ to $A^-$ and edges connecting $B^+$ to $B^-$. So, since $G(D_A,D_B \,|\, C_2)$ is a subgraph of $G(D_A,D_B \,|\, C_1, C_2)$, the graph $G(D_A,D_B \,|\, C_1, C_2)$ also has edges connecting $A^+$ to $A^-$ and edges connecting $B^+$ to $B^-$. This implies $G(D_A,D_B \,|\, C_1, C_2)$ has the form of Figure \ref{Dist3Fig8b}a.

It remains to establish that $c > a+b > 0$ and $d > a+b > 0$ in Figure \ref{Dist3Fig8b}a. To see this, observe that $a+b$ in Figure \ref{Dist3Fig8b}a is equal to the number of connections in $(C_1 \cup C_2) \cap F_P^+$. Also observe that any connection in $C_1 \cap F_P^+$ is properly isotopic in $F_P^+$ to one of the four connections $\delta_1$, $\delta_2$, $\delta_3$, $\delta_4$ shown in Figure \ref{Dist3Fig22b}. Therefore, since $m-1 > n > 1$, $\delta_1$, $\delta_2$, $\delta_3$ and $\delta_4$ intersect $\partial D_P$ respectively $m > 3$, $n > 1$, $m+n > 4$ and $m-n > 1$ times. In particular, each connection $\delta \in C_1 \cap F_P^+$ satisfies $|\delta \cap \partial D_P| \geq 2$. Since each connection $\delta \in C_2 \cap F_P^+$ also satisfies $|\delta \cap \partial D_P| \geq 2$, we have $c \geq a+b > 0$ and $d \geq a+b > 0$. However, there is also a connection $\omega_1 \in C_1 \cap F_P^+$ with $|\omega_1 \cap \partial D_P| > 2$. This implies $c > a+b > 0$ and $d > a+b > 0$, which is what we need. This proves Claim \ref{DA, DB is a set of SUMS}. \end{proof}
And also completes the proof of Theorem \ref{4 pairs of rectangles}. 
\end{proof}

\begin{lem} \label{forms of graphs of curves bounding disks and disjoint from disks}
Suppose $W$ is a genus two handlebody with a complete set of cutting disks $\{D_S,D_T\}$, $C_1$ and $C_2$ are a pair of disjoint nonseparating simple closed curves in $\partial W$ such that $C_1$ and $C_2$ have only essential intersections with $D_S$ and $D_T$, and $C_1$ bounds a disk in $W$. Let $G(D_S,D_T \,|\, C_1)$ and  $G(D_S,D_T \,|\, C_2)$ be the graphs of the Heegaard diagrams of $C_1$ and $C_2$ respectively with respect to $D_S$ and $D_T$. Then either:
\begin{enumerate}
\item $G(D_S,D_T \,|\, C_1)$ has the form of Figure \emph{\ref{Dist3Fig8b}c} with $b > 0$;
\item $C_1$ is isotopic to $\partial D_S$ or $\partial D_T$;
\item $C_1$ is a bandsum of $\partial D_S$ with $\partial D_T$, so $C_1$ appears as a simple closed curve which separates vertex $S^+$ from vertex $S^-$ and vertex $T^+$ from vertex $T^-$ in $G(D_S,D_T \,|\, C_1)$.
\end{enumerate}
In any case, either there are no edges in $G(D_S,D_T \,|\, C_2)$ connecting $S^+$ to $S^-$, or there are no edges in $G(D_S,D_T \,|\, C_2)$ connecting $T^+$ to $T^-$.
\end{lem}

\begin{proof}
Let $D_C$ be the disk which $C_1$ bounds in $W$, and suppose $C_1$ has essential intersections with $\partial D_S \cup \partial D_T$. Then $D_C \cap (D_S \cup D_T)$ is nonempty, and we may assume $D_C \cap (D_S \cup D_T)$ consists only of arcs of intersection. So there is an arc of intersection in $D_C \cap (D_S \cup D_T)$ which cuts off an outermost subdisk $D$ of $D_C$. Let $\omega$ be the arc $\partial D \cap C_1$. Then $\omega$ has both of its endpoints at the same vertex of $G(D_S,D_T \,|\, C_1)$. It follows that, in this case, $G(D_S,D_T \,|\, C_1)$ has the form of Figure~\ref{Dist3Fig8b}c with $b > 0$. On the other hand, if $C_1$ does not have essential intersections with $\partial D_S \cup \partial D_T$, it is clear that (2) or (3) holds. 

Turning to the form of $G(D_S,D_T \,|\, C_2)$, we see that if $C_1$ has essential intersections with $\partial D_S \cup \partial D_T$, and $\omega$ has both endpoints at $S^+$ or $S^-$ (resp $T^+$ or $T^-$), then there are no edges in $G(D_S,D_T \,|\, C_2)$ connecting $T^+$ to $T^-$ (resp $S^+$ to $S^-$).

On the other hand, if $C_1$ does not have essential intersections with $\partial D_S \cup \partial D_T$, so (2) or (3) holds, then again, either there are no edges in $G(D_S,D_T \,|\, C_2)$ connecting $S^+$ to $S^-$, or there are no edges in $G(D_S,D_T \,|\, C_2)$ connecting $T^+$ to $T^-$.
\end{proof}

\begin{lem} \label{c > a+b > 0 and d > a+b > 0 imply minimal}
Suppose $W$ is a genus two handlebody with a complete set of cutting disks $\{D_S,D_T\}$, and $\{C_1,C_2\}$ is a pair of disjoint essential simple closed curves in $\partial W$ such that the graph $G(D_S,D_T\,|\,C_1,C_2)$ of the Heegaard diagram of $C_1$ and $C_2$ with respect to $D_S$ and $D_T$ has the form of Figure \emph{\ref{Dist3Fig8b}a}, with $c > a+b > 0$ and $d > a+b > 0$. Then the set of cutting disks $\{D_S,D_T\}$ of $W$ is the one and only complete set of cutting disks of $W$ intersecting $C_1\cup C_2$ minimally.
\end{lem}

\begin{proof}
Suppose, to the contrary, that there exists a complete set of cutting disks $\{D_1,D_2\}$ of $W$, with $\{D_1,D_2\}$ not isotopic to $\{D_S,D_T\}$ in $W$, such that the complexity of the Heegaard diagram $D(D_1,D_2\,|\,C_1,C_2)$ is less than or equal to the complexity of $D(D_S,D_T\,|\,C_1,C_2)$. Suppose furthermore, as we may, that among such complete sets of cutting disks of $H$, $\{D_S,D_T\}$ and $\{D_1,D_2\}$ intersect minimally.

If $\{D_S,D_T\}$ and $\{D_1,D_2\}$ are disjoint, then one of $\{D_1,D_2\}$, say $D_1$, is a bandsum of $D_S$ and $D_T$ in $W$, and 
$$|D_1 \cap (C_1\cup C_2)| \leq \max\{|D_A \cap (C_1\cup C_2)|,|D_B \cap (C_1\cup C_2)|\}.$$
However, because $c > a+b > 0$ and $d > a+b > 0$ in the graph $G(D_S,D_T\,|\,C_1, C_2)$, this is impossible. So $\{D_S,D_T\}$ and $\{D_1,D_2\}$ must have essential intersections.

We may assume disks in $\{D_S,D_T\}$ intersect disks in $\{D_1,D_2\}$ only in arcs. So some disks in $\{D_S,D_T,D_1,D_2\}$ contain outermost subdisks cut off by outermost arcs of intersection of disks in $\{D_S,D_T\}$ with disks in $\{D_1,D_2\}$. Among the set of outermost subdisks of the disks in $\{D_S,D_T,D_1,D_2\}$, let $D$ be one that intersects $C_1\cup C_2$ minimally. 

Note that if $D$ were a subdisk of $D_S$ or $D_T$, then $D$ could be used to perform surgery on one of $D_1$, $D_2$, leading to a contradiction of the assumed minimality properties of $\{D_1,D_2\}$ vis a vis $C_1\cup C_2$ or $\{D_S,D_T\}$. So $D$ must be a subdisk of $D_1$ or $D_2$. But then $D$ could be used to perform surgery on one of $\{D_S,D_T\}$, say $D_S$, yielding a cutting disk $D_S'$ of $W$, disjoint from $D_S$ and $D_T$, i.e. a bandsum of $D_S$ and $D_T$, such that 
$$|D_S' \cap (C_1\cup C_2)| \leq |D_S \cap (C_1\cup C_2)|.$$
However, as before, since $c > a+b > 0$ and $d > a+b > 0$ in the graph $G(D_S,D_T\,|\,C_1, C_2)$, this is impossible.
\end{proof}

\begin{cor} \label{Sets of SUMS exist}
The complete set of cutting disks $\{D_A,D_B\}$ of the handlebody $H$ in the Heegaard splittings described by the R-R diagrams in Figures \emph{\ref{Dist3Fig9a}} and \emph{\ref{Dist3Fig9b}} is a set of SUMS. 
\end{cor}

\begin{proof}
First, recall, as mentioned before, that a rectangle $R_{pq}$ with $p \in \{a,b\}$, $q \in \{x,y\}$, and $|R_{pq}| = e$ exists in the Heegaard surface of the R-R diagram in Figure~\ref{Dist3Fig9a} (resp Figure~\ref{Dist3Fig9b}), if and only if $\partial D_Q$ intersects the face of the $P$-hexagon in Figure~\ref{Dist3Fig9a} (resp Figure~\ref{Dist3Fig9b}), with label $e$ in two adjacent points.

Then examination of Figures \ref{Dist3Fig9a} and \ref{Dist3Fig9b} shows that, in each case, there are four pairs of rectangles in the Heegaard surface which satisfy the hypothesis of Theorem \ref{4 pairs of rectangles}. It follows that, in each case, the set of cutting disks $\{D_A,D_B\}$ of the underlying handlebody $H$ is a set of SUMS.
\end{proof}

\begin{cor} \label{Splittings are distance three}
The Heegaard splittings in the R-R diagrams of Figures \emph{\ref{Dist3Fig9a}} and \emph{\ref{Dist3Fig9b}} are distance three splittings.
\end{cor}

\begin{proof}
As, in Corollary \ref{Sets of SUMS exist}, examination of Figures \ref{Dist3Fig9a} and \ref{Dist3Fig9b} shows that, in each case, there are four pairs of rectangles in the Heegaard surface which satisfy the hypothesis of Theorem \ref{4 pairs of rectangles}. It follows that the Heegaard splittings described by the R-R diagrams in Figures \ref{Dist3Fig9a} and \ref{Dist3Fig9b} do not have the DCP, and so they are splittings of distance at least three.

On the other hand, Figures \ref{Dist3Fig9c} and \ref{Dist3Fig9d} show that, in each splitting, there exist disjoint nonseparating simple closed curves $\alpha$ and $\beta$ in the Heegaard surface which are $(SF,PP)$ pairs. This implies these splittings have distance at most three. Hence each of these splittings is a distance three splitting.
\end{proof}

\begin{figure}[htbp]
\centering
\includegraphics[width = .50\textwidth]{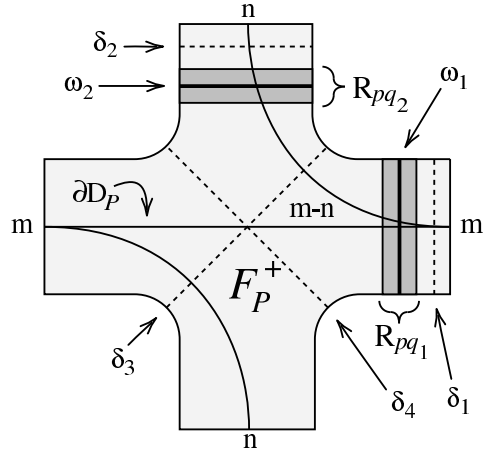}
\caption{Here $F_P^+$ is one of the two once-punctured tori in $\Sigma$ with boundary $\Gamma$. The figure shows $F_P^+$ cut open along a pair of properly embedded arcs parallel to edges of rectangles $R_{pq_1}$ and $R_{pq_2}$ in $\mathcal{R}_{pq}$ where $|R_{pq_1}|$ = $|R_{pq_1} \cap \partial D_P|$ = $m$, $|R_{pq_2}|$ = $|R_{pq_2} \cap \partial D_P|$ = $n$, and $m-1>n>1$. (It is always the case that $\gcd(m,n) = 1$).}
\label{Dist3Fig22b}
\end{figure}

\begin{section}{Deriving an R-R diagram of the second splitting of $M$ from the R-R diagram of the first splitting of $M$} \label{Obtaining the second splitting}

\begin{figure}[htbp]
\centering
\includegraphics[width = 0.85\textwidth]{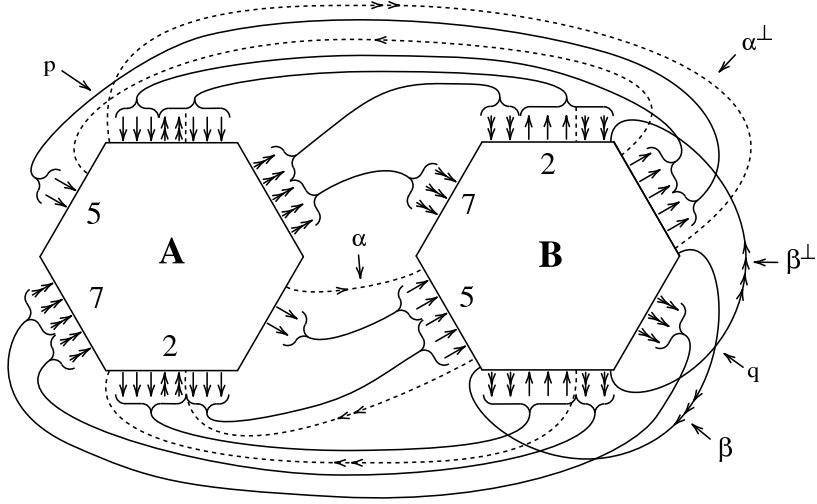}
\caption{The R-R diagram of the first splitting of $M$ in Figure~\ref{Dist3Fig9a} with four curves $\alpha$, $\alpha^\perp$, $\beta$ and $\beta^\perp$ added to the diagram. Here $\alpha$, $\alpha^\perp$, $\beta$ and $\beta^\perp$ represent $A^5B^5$, $b^5A^2B^2A^2$, $B^5$, and $B^2$ respectively in $\pi_1(H)$, while they represent $X^3$, $X^2$, $Y^5x^3Y^2$, and $x^5y^3$ respectively in $\pi_1(H')$. So $(\alpha,\beta)$ is a $(SF,PP)$ pair in $H$, and a $(PP,SF)$ pair in $H'$.}
\label{Dist3Fig9c}
\end{figure}

In order to obtain an R-R diagram of the second splitting of $M$ from the first splitting of $M$, we carry out the following three steps.

\begin{enumerate}
\item Obtain a geometric 4-generator, 4-relator presentation $\mathcal{P}$ of $\pi_1(M)$ from the diagram of the first splitting of $M$ in Figure \ref{Dist3Fig9a}.
\item Reduce the presentation $\mathcal{P}$ of step (1) to a 2-generator, 2-relator geometric presentation which has minimal length under automorphisms.
\item Produce an R-R diagram realizing the presentation $\mathcal{P}$ obtained in step (2).
\end{enumerate}

Figure \ref{Dist3Fig9c} shows the R-R diagram of the original splitting of $M$ in Figure \ref{Dist3Fig9a} with four simple closed curves $\alpha$, $\alpha_\perp$, $\beta$ and $\beta_\perp$ added to the diagram so that $\alpha$, $\alpha^\perp$, $\beta$ and $\beta^\perp$ represent $A^5B^5$, $b^5A^2B^2A^2$, $B^5$, and $B^2$ respectively in $\pi_1(H)$, while they represent $X^3$, $X^2$, $Y^5x^3Y^2$, and $x^5y^3$ respectively in $\pi_1(H')$. Thus $(\alpha,\beta)$ is a $(SF,PP)$ pair in $H$, and a $(PP,SF)$ pair in $H'$.

(Note that we adopt the space-saving convention of using pairs of uppercase and lowercase letters to denote generators and their inverses in free groups and the relators of presentations. So if $x$ is a generator of a free group, then $X = x^{-1}$.)

The curve $\beta^\perp$ has been chosen so that $\beta^\perp$ is disjoint from the separating curve $\Gamma$ in $\Sigma$, and so that $\beta$ and $\beta^\perp$ intersect transversely in a single point $q$. Then, in particular, $\beta$ and $\beta^\perp$ lie completely on the B-handle of Figure \ref{Dist3Fig9c}. 

The curve $\alpha^\perp$ has been chosen so that the pair $(\alpha,\alpha^\perp)$ has  properties with respect to the handlebody $H'$ analogous to those enjoyed by the pair $(\beta,\beta^\perp)$ with respect to $H$. That is: $\alpha$ and $\alpha^\perp$ intersect transversely once in a single point $p$, and they are both disjoint from a separating curve $\Gamma'$ in $\Sigma$ such that $\Gamma'$ bounds a disk in $\partial H'$ separating the cutting disks $D_X$ and $D_Y$ of $H'$. (Note $\Gamma'$ is not shown in Figure~\ref{Dist3Fig9c}.)

\subsection{Obtaining a genus four splitting of $\boldsymbol{M}$} \hfill
\smallskip

The separating disk in $H$ which $\Gamma$ bounds cuts $H$ into two solid tori $V_A$ and $V_B$, with meridional disks $D_A$ and $D_B$ respectively. Similarly, the separating curve $\Gamma'$, cuts $H'$ into two solid tori $V_X$ and $V_Y$, with meridional disks $D_X$ and $D_Y$ respectively. 

Let $v_B$ be a regular neighborhood in $V_B$ of a core of $V_B$. 
Then $V_B \setminus \text{int}(v_B)$ is homeomorphic to $\partial V_B \boldsymbol{\times} I$. And if $D_q$ is a small disk in $\partial V_B$ containing the point $q = \beta \cap \beta^\perp$, then $v_B$ can be attached to $H'$ by the one-handle $D_q \boldsymbol{\times} I$.

In similar fashion, let $v_X$ be a regular neighborhood in $V_X$ of a core of $V_X$. 
Then $V_X \setminus \text{int}(v_X)$ is homeomorphic to $\partial V_X \boldsymbol{\times} I$. And if $D_p$ is a small disk in $\partial V_X$ containing the point $p = \alpha \cap \alpha^\perp$, then $v_X$ can be attached to $H$ by the one-handle $D_p \boldsymbol{\times} I$. These two changes transform $H$ and $H'$ into genus four handlebodies $H_4$ and $H'_4$ giving a genus four Heegaard splitting of $M$.

\subsection{Locating complete sets of cutting disks of {\mathversion{bold}$H_4$} and {\mathversion{bold}$H'_4$}} \hfill
\smallskip

The next step is to locate complete sets of cutting disks of $H_4$ and $H'_4$. First, note that $D_X \cap v_X$ is a meridional disk of $v_X$. Next, note that $\beta \boldsymbol{\times} I$ and $\beta^\perp \boldsymbol{\times} I$ are annuli in $\partial V_B \boldsymbol{\times} I$, which become cutting disks $D_\beta$ and $D_{\beta^\perp}$ of $H_4$ when the one handle $D_q \boldsymbol{\times} I$ is attached to $H'$. Also note that $D_A$ is still a cutting disk of $H_4$. It follows that $H_4$ has a complete set of cutting disks comprised of $D_A$, $D_\beta$, $D_{\beta^\perp}$, and $D_X \cap v_X$.
Similarly, $H'_4$ has a complete set of cutting disks comprised of $D_Y$, $D_\alpha$, $D_{\alpha^\perp}$, and $D_B \cap v_B$.

\subsection{Obtaining an initial geometric presentation of {\mathversion{bold}$\pi_1(M)$}} \hfill
\smallskip

If we take generators $A$, $C$, $D$, and $E$ of $\pi_1(H_4)$ which are dual in $H_4$ to $D_A$, $D_\beta$, $D_{\beta^\perp}$, and $D_X \cap v_X$ respectively, then $\pi_1(M)$ has the geometric presentation
\begin{equation} \label{abstract pres}
\mathcal{P} = \langle A,C,D,E \,|\, \partial D_\alpha, \partial D_{\alpha^\perp},  \partial (D_B \cap v_B), \partial D_Y \rangle.
\end{equation}
The next step is to express the abstract relators of (\ref{abstract pres}) as cyclic words in the generators $A$, $C$, $D$, and $E$ of $\mathcal{P}$. 

To obtain the cyclic word which $\partial D_\alpha$ represents in $\pi_1(H_4)$, start at the point $p = \alpha \cap \alpha^\perp$ and proceed around $\alpha$ recording the oriented intersections of $\alpha$ with $\beta$, $\beta^\perp$, and $\partial D_A$, while ignoring intersections of $\alpha$ with $\partial D_X$ until returning to $p$. Then starting at $p$, retrace $\alpha$ in the opposite direction and record only the oriented intersections of $\alpha$ with $\partial D_X$ until returning to $p$. This yields $\partial D_\alpha$ = $A^5De^3$.

The cyclic word which $\partial D_{\alpha^\perp}$ represents in $\pi_1(H_4)$ can be obtained in the same way as that of $\partial D_\alpha$. This yields $\partial D_{\alpha^\perp}$ = $dA^2cA^2e^2$. (This works for $\partial D_\alpha$ and $\partial D_{\alpha^\perp}$ because the annulus in $D_X$ bounded by $\partial D_X$ and $\partial (D_X \cap v_X)$ lies in $H_4$.) 

It is fairly easy to find the cyclic word which $\partial (D_B \cap v_B)$ represents in $\pi_1(H_4)$ because $\partial (D_B \cap v_B)$ and $\partial D_B$ bound an annulus in $D_B$ which lies in $H_4$, and the curves $\beta$ and $\beta^\perp$ form a basis for a once-punctured torus in $\partial H_4$ in which $\partial D_B$ lies. It follows that the cyclic word which $\partial (D_B \cap v_B)$ represents in $\pi_1(H_4)$ can be obtained by traversing $\partial D_B$ in $\Sigma$ while recording the oriented intersections of $\partial D_B$ with $\beta$ and $\beta^\perp$. This yields $\partial (D_B \cap v_B)$ = $DC^2DC^3$.

Finally, it is straightforward to obtain the cyclic word which $\partial D_Y$ represents in $\pi_1(H_4)$ by proceeding around $\partial D_Y$ in Figure \ref{Dist3Fig9c} while recording the oriented intersections of $\alpha$ with $\beta$, $\beta^\perp$ and $\partial D_A$ in terms of the new set of generators $A$, $C$ and $D$ of $\pi_1(H_4)$. This yields $\partial D_Y$ = $A^7Dc(A^7DcA^7cA^2c)^2$.

Putting these pieces together yields the presentation
\begin{equation} \label{pres 1}
\mathcal{P} = \langle A,C,D,E \,|\, A^5De^3, dA^2cA^2e^2, DC^2DC^3, A^7Dc(A^7DcA^7cA^2c)^2 \rangle.
\end{equation}

\subsection{Reducing $\boldsymbol{\mathcal{P}}$ to a 2-generator, 2-relator geometric presentation} \hfill
\smallskip

Suppose $(\Sigma;V,W)$ is a Heegaard splitting, with $\Sigma$ bounding handlebodies $V$ and $W$. Recall that the splitting $(\Sigma;V,W)$ has a \emph{trivial handle} if there exist cutting disks $D_v$ of $V$ and $D_w$ of $W$ such that $D_v$ and $D_w$ intersect transversely in a single point. 

In order to obtain the second genus two splitting of $M$ from the genus four splitting by  $H_4$ and $H'_4$, we want to find two independent trivial handles, one involving  $D_\alpha$, and one involving $D_\beta$. Since $D_{\beta^\perp}$ is a disk in $H_4$, which meets $D_\alpha$ transversely at a single point, $D_\alpha$ and $D_{\beta^\perp}$ can be used as the first trivial handle. 

To eliminate the $(D_\alpha,D_{\beta^\perp})$ trivial handle, the cutting disks of $H'_4$, other than $D_\alpha$, need to have their intersections with $D_{\beta^\perp}$ removed by forming bandsums of these disks with $D_\alpha$ along arcs of $\partial D_{\beta^\perp}$. Algebraically, this amounts to changing presentation (\ref{pres 1}) by solving $\partial D_\alpha = A^5De^3 = 1$ for $D$, from which $D = a^5E^3$, then replacing all occurrences of $D$ in the other three relators of (\ref{pres 1}) with $a^5E^3$, and then dropping the generator $D$ and the relator $A^5De^3$ from (\ref{pres 1}). This changes (\ref{pres 1}) into the geometric presentation
\begin{equation} \label{pres 2}
\mathcal{P} = \langle A,C,E \,|\, A^7cA^2e^5, a^5E^3C^2a^5E^3C^3, A^2E^3c(A^2E^3cA^7cA^2c)^2 \rangle.
\end{equation}

Topologically, eliminating the trivial handle $(D_\alpha,D_{\beta^\perp})$, has \emph{destabilized} the genus four Heegaard splitting of $M$ by $H_4$ and $H'_4$ by turning it into a genus three Heegaard splitting of $M$ by genus three handlebodies $H_3$ and $H'_3$. Then, in this genus three splitting, $D_\beta$ and the cutting disk of $H'_3$ whose boundary represents the first relator $A^7cA^2e^5$ of (\ref{pres 2}) form a trivial handle. 

Algebraically, eliminating this trivial handle amounts to solving $A^7cA^2e^5 = 1$ in (\ref{pres 2}) for $C$, from which $C = A^2e^5A^7$ and $c = a^7E^5a^2$, then replacing $C$ and $c$ in the other two relators of (\ref{pres 2}) with $A^2e^5A^7$ and $a^7E^5a^2$ respectively, and then dropping the generator $C$ and the relator $A^7cA^2e^5$ from (\ref{pres 2}). This turns (\ref{pres 2}) into the geometric presentation 
\begin{equation}\label{pres 3}
\mathcal{P} = \langle A,E \,|\, A^9e^5(A^2E^3A^2e^5A^9e^5)^2, E^8a^7(E^8a^7E^5a^2E^5a^7)^2 \rangle.
\end{equation}

Finally, taking the inverse of the first relator, and replacing $E$ with $A$ and $a$ with $B$ in (\ref{pres 3}), turns (\ref{pres 3}) into
\begin{equation} \label{second pres}
\mathcal{P} = \langle A,B \,|\, A^8B^7(A^8B^7A^5B^2A^5B^7)^2, A^5B^9(A^5B^9A^5B^2a^3B^2)^2 \rangle. 
\end{equation}

\subsection{Obtaining an R-R diagram of a realization of $\boldsymbol{\mathcal{P}}$} \hfill
\smallskip

At this point, we have a genus two Heegaard splitting of $M$ by handlebodies $H_2$ and $H_2'$, which was obtained by eliminating a trivial handle from the genus three splitting of $M$ by $H_3$ and $H_3'$. Next, for simplicity, we drop the subscripts from $H_2$ and $H_2'$. So we have a genus two Heegaard surface $\Sigma$ bounding handlebodies $H$ and $H'$. Then there exist complete sets $\{D_A',D_B'\}$ and $\{D_X,D_Y\}$ of cutting disks of $H$ and $H'$ respectively such that the Heegaard diagram $D(D_A',D_B'\,|\,\partial D_X,\partial D_Y)$ of $\partial D_X$ and $\partial D_Y$ with respect to $D_A'$ and $D_B'$ realizes (\ref{second pres}). 

Note that, in general, there is nothing unique about Heegaard diagrams realizing (\ref{second pres}) since diagrams such as $D(D_A',D_B'\,|\,\partial D_X,\partial D_Y)$ can be modified by Dehn twists of $\partial D_X$ and $\partial D_Y$ about simple closed curves which bound disks in $H$. The following claim provides a way to deal with this minor annoyance. 

\begin{claim} \label{Min complexity realization}
There exists a complete set of cutting disks $\{D_A,D_B\}$ of $H$ such that the Heegaard diagram $D(D_A,D_B\,|\,\partial D_X,\partial D_Y)$ has both minimal complexity and 
realizes presentation \emph{(\ref{second pres})}.
\end{claim}

\begin{proof} [Proof of Claim \emph{\ref{Min complexity realization}}]
Lemma \ref{min length under auts} shows (\ref{second pres}) has minimal length under automorphisms of the free group $F(A,B)$, and it also shows $\{A,B\}$ is the only basis of $F(A,B)$ in which (\ref{second pres}) has minimal length. It then follows from the main result of \cite{Z1} that $H$ has a unique set of cutting disks $\{D_A,D_B\}$ such that $D(D_A,D_B\,|\,\partial D_X,\partial D_Y)$ has both minimal complexity and realizes (\ref{second pres}).
\end{proof}

\begin{claim} \label{Graph of min complexity realization}
The graph $G(D_A,D_B\,|\,\partial D_X,\partial D_Y)$ underlying $D(D_A,D_B\,|\,\partial D_X,\partial D_Y)$ has the form of Figure \emph{\ref{Dist3Fig8b}a}, with $c, d > 0$.
\end{claim}

\begin{proof} [Proof of Claim \emph{\ref{Graph of min complexity realization}}]
By Lemma \ref{3 types of graphs which occur}, $G(D_A,D_B\,|\,\partial D_X,\partial D_Y)$ has the form of one of the three graphs in Figure \ref{Dist3Fig8b}. Since $D(D_A,D_B\,|\,\partial D_X,\partial D_Y)$ has minimal complexity, $G(D_A,D_B\,|\,\partial D_X,\partial D_Y)$ does not have the form of Figure \ref{Dist3Fig8b}c. And, because both generators appear in the relators of (\ref{second pres}) with exponents having absolute value greater than one, $G(D_A,D_B\,|\,\partial D_X,\partial D_Y)$ does not have the form of Figure \ref{Dist3Fig8b}b. It follows $G(D_A,D_B\,|\,\partial D_X,\partial D_Y)$ has the claimed form.
\end{proof}

Finally, we can produce an R-R diagram $\mathcal{D}$ with $D(D_A,D_B\,|\,\partial D_X,\partial D_Y)$ as its underlying Heegaard diagram. R-R diagrams with  $D(D_A,D_B\,|\,\partial D_X,\partial D_Y)$ as their underlying Heegaard diagram are parametrized by the isotopy class of an essential separating simple closed curve $\Gamma$ in the Heegaard surface $\Sigma$ which is disjoint from $D_A$ and $D_B$. Since $G(D_A,D_B\,|\,\partial D_X,\partial D_Y)$ has the form of Figure~\ref{Dist3Fig8b}a, it is natural to take $\Gamma$ to be the unique separating simple closed curve in $\Sigma$, which is disjoint from $\partial D_A$ and $\partial D_B$, and is also disjoint from any edge of $D(D_A,D_B\,|\,\partial D_X,\partial D_Y)$ connecting  $D_A^+$ to $D_A^-$ or $D_B^+$ to $D_B^-$. Thus $\Gamma$ becomes the essential separating simple closed curve in $\Sigma$, disjoint from $\partial D_A$ and $\partial D_B$, which intersects $\partial D_X$ and $\partial D_Y$ minimally.

\smallskip
Then in order to determine the form of $\mathcal{D}$ we need to:
\begin{enumerate}
\item Determine the labels for each face of the hexagons representing the A-handle and B-handle of $\mathcal{D}$;
\item Determine how the cyclic orders of the labeling of the faces of the A-handle hexagon and the B-handle hexagon need to be coordinated;
\item Determine how many connections lie in each isotopy class of connections on each handle of $\mathcal{D}$;
\item Determine how to connect endpoints of connections on the A-handle with endpoints of connections on the B-handle of $\mathcal{D}$.
\end{enumerate}

All of these items can be determined by scanning the relators of $\mathcal{P}$.  For example, suppose $G \in \{A,B\}$. Then, since $\Gamma$ is disjoint from edges of $D(D_A,D_B\,|\,\partial D_X,\partial D_Y)$ connecting $D_A^+$ to $D_A^-$ or $D_B^+$ to $D_B^-$, $G$ must appear in the relators of $\mathcal{P}$ with exponents having at most three absolute values, say  $e_1$, $e_2$ and $e_3$, with $e_2 = e_1+e_3$, and then the labels of the faces of the $G$-hexagon must be in clockwise cyclic order either $(e_1, e_2,e_3,-e_1,-e_2,-e_3)$ or $(e_3, e_2,e_1,-e_3,-e_2,-e_1)$.

Next, scanning the relators of $\mathcal{P}$ shows $A$ appears with exponents having absolute values $3$, $8$ and $5$, while $B$ appears with exponents having absolute values $2$, $9$ and $7$. Up to orientation reversing homeomorphism of $\mathcal{D}$, the cyclic order of the labeling of the faces of one of the handles of $\mathcal{D}$ may be chosen arbitrarily. So we may label the faces of the A-handle of $\mathcal{D}$ with $(5,8,3,-5,-8,-3)$ in clockwise cyclic order. Then, with the cyclic order of the labels of the A-handle specified, it is easy to see that, because $(A^8B^7)^{\pm1}$, $(B^7A^8)^{\pm1}$ and $(A^5B^2)^{\pm1}$ appear in $\mathcal{P}$, the faces of the B-hexagon of $\mathcal{D}$ must be labeled in clockwise cyclic order $(7,9,2,-7,-9,-2)$ if $\mathcal{D}$ is to realize $\mathcal{P}$.

Next, let $|G^{\pm e}|$, for $G \in \{A,B\}$, denote the total number of appearances of $G$ in the relators of $\mathcal{P}$ with exponent having absolute value $e$. Then $|A^{\pm5}| = 9$, $|A^{\pm8}| = 3$, $|A^{\pm3}| = 2$, $|B^{\pm7}| = 5$, $|B^{\pm9}| = 3$, $|B^{\pm2}| = 6$, and clearly these values determine the number of connections in each isotopy class of connections on the two handles of $\mathcal{D}$.

It remains to determine how edges of $\mathcal{D}$ connect endpoints of connections on the A-hexagon of $\mathcal{D}$ with endpoints of connections on the B-hexagon of $\mathcal{D}$. In general, each 2-syllable subword of the relators of $\mathcal{P}$ of the form $(A^mB^n)^{\pm1}$ with $m \in \{\pm5,\pm8,\pm3\}$ and $n \in \{\pm7,\pm9,\pm2\}$ corresponds to an edge of $\mathcal{D}$ which connects an endpoint of a connection in the face of the A-hexagon of $\mathcal{D}$ with label $-m$ to an endpoint of a connection in the face of the B-hexagon of $\mathcal{D}$ with label $n$. In particular, since $|A^{\pm5}| = 9$, there must be a set $\mathcal{S}$ of 9 edges in $\mathcal{D}$ connecting the 9 endpoints of connections in the face of the A-hexagon of $\mathcal{D}$ with label $-5$ to 9 consecutive endpoints of connections in the boundary of the B-hexagon of $\mathcal{D}$. 

One can see where these 9 consecutive endpoints of edges in $\mathcal{S}$ are located in the boundary of the B-hexagon in the following way. For each $e \in \{\pm 7,\pm 9, \pm2\}$, let $|(A^5B^e)^{\pm1}|$ denote the total number of appearances of two syllable subwords in the relators of $\mathcal{P}$ of the form $(A^5B^e)^{\pm1}$. A scan of the relators of $\mathcal{P}$ shows the only nonzero values in this set are: $|(A^5B^7)^{\pm1}| = 2$, $|(A^5B^9)^{\pm1}| = 3$ and $|(A^5B^2)^{\pm1}| = 4$. This implies the edges of $\mathcal{S}$ must appear in $\mathcal{D}$ so that $\mathcal{S}$ is the disjoint union of subsets of 2, 3 and 4 edges which meet the faces of the B-hexagon of $\mathcal{D}$ with labels 7, 9 and 2 respectively. It is easy to see there is only one way to do this. And this, in turn, implies that how edges of $\mathcal{D}$ connect endpoints of connections on the A-hexagon of $\mathcal{D}$ with endpoints of connections on the B-hexagon of $\mathcal{D}$ is completely determined.

The resulting R-R diagram $\mathcal{D}$ appears in Figure \ref{Dist3Fig9b}. One checks easily that $\mathcal{D}$ realizes (\ref{second pres}). And, by construction, it is an R-R diagram of a genus two Heegaard splitting of $M$.
\end{section}

\begin{lem}
\label{min length under auts}
Suppose $\mathcal{P} = \langle \, A,B \,|\, \mathcal{R}_1, \mathcal{R}_2 \, \rangle$ is a two-generator, two-relator presentation in which both $\mathcal{R}_1$ and $\mathcal{R}_2$ are cyclically reduced cyclic words.  If there is an automorphism of the free group $F(A,B)$ generated by $A$ and $B$ which reduces the length of $\mathcal{P}$, then one of the four Whitehead automorphisms
\begin{gather}
(A,B) \mapsto (AB^{\pm1},B) \\
(A,B) \mapsto (A,BA^{\pm1})
\end{gather}
will also reduce the length of $\mathcal{P}$.
\end{lem}

\begin{proof}
This is a well-known result of Whitehead. See \cite{W} or \cite{LS}.
\end{proof}

\subsection{The two genus two splittings of $\boldsymbol{M}$ are not homeomorphic} 

\begin{prop} \label{splittings are not homeomorphic}
The genus two Heegaard splittings of $M$ determined by the R-R diagrams of Figures \emph{\ref{Dist3Fig9a}} and \emph{\ref{Dist3Fig9b}} are not homeomorphic.
\end{prop}

\begin{proof}
Consider how the simple closed curves $\partial D_P$, for $P \in \{A,B\}$ intersect the bands of connections $F_P \cap (\partial D_X \cup \partial D_Y)$ in $F_P$ in Figure \ref{Dist3Fig9a} and Figure~\ref{Dist3Fig9b}. It is not hard to see that, in both Figures \ref{Dist3Fig9a} and \ref{Dist3Fig9b}, the graph $G(D_X,D_Y\,|\, \partial D_A, \partial D_B)$ of the Heegaard diagram $D(D_X,D_Y\,|\, \partial D_A, \partial D_B)$ of $\partial D_A$ and $\partial D_B$ with respect to $D_X$ and $D_Y$ has the form of Figure \ref{Dist3Fig8b}a with $c > a+b > 0$ and $d > a+b > 0$. It follows from Lemma \ref{c > a+b > 0 and d > a+b > 0 imply minimal}, that, in both diagrams, the set of cutting disks $\{D_X,D_Y\}$ of $H'$ is the unique complete set of cutting disks of $H'$ intersecting $\partial D_A \cup \partial D_B$ minimally. 

Next, since Corollary \ref{Sets of SUMS exist} shows that, in both Figure \ref{Dist3Fig9a} and Figure \ref{Dist3Fig9b}, $\{D_A,D_B\}$ is a set of SUMS, it follows that, in both Figures \ref{Dist3Fig9a} and Figure \ref{Dist3Fig9b}, the set of simple closed curves $\{\partial D_A, \partial D_B, \partial D_X, \partial D_Y\}$ form the unique minimal complexity Heegaard diagram carried by the splitting surface $\Sigma$. However, in the Heegaard surface $\Sigma$ of Figure \ref{Dist3Fig9a}, $|(\partial D_A \cup \partial D_B) \cap (\partial D_X \cup \partial D_Y)|$ = 121, while in the Heegaard surface $\Sigma$ of Figure~\ref{Dist3Fig9b}, $|(\partial D_A \cup \partial D_B) \cap (\partial D_X \cup \partial D_Y)|$ = 149. Since these minimal complexities differ, the Heegaard surfaces of $M$ in Figures \ref{Dist3Fig9a} and \ref{Dist3Fig9b} are not homeomorphic.
\end{proof}

\end{document}